\documentclass[11pt]{article}
\usepackage[mathscr]{eucal}
\usepackage{amsthm,amsmath}
\usepackage[normalem]{ulem}

\DeclareMathAlphabet{\mathpzc}{OT1}{pzc}{m}{it}

\newcommand{{\M}}{{\textsf{{L}$_\sharp^{2}$}}}
\newcommand{\bsfV}{\boldsymbol{\mathsf V}}
\newcommand{\bsfU}{\boldsymbol{\mathsf U}}

\newcommand{\bsfu}{\boldsymbol{\mathsf u}}

\usepackage{bbm}
\usepackage{dsfont}
\usepackage{color}
\newcommand{\db}[1]{\color{red}{#1}\color{black}}
\usepackage{mathrsfs}
\usepackage{amssymb}
\usepackage{calligra}
\usepackage{fontenc}
\usepackage{amsbsy}
\usepackage{comment}
\usepackage{graphicx}
\graphicspath{%
    {converted_graphics/}
    {/}
}

\newcommand{\Div}{\mbox{\rm div}\,}

\newcommand{\supp}{\mbox{\rm supp}\,}
\newcommand{\curl}{\mbox{\rm curl}\,}

\newcommand{\Int}[2]{{\displaystyle \int_{ #1}^{ #2}}}
\newcommand{\Lim}[1]{{\displaystyle \lim_{ #1}}}

\newcommand{\Sum}[2]{{\displaystyle \sum_{#1}^{#2}}}

\newcommand{\BCup}[2]{{\displaystyle \bigcup_{#1}^{#2}}}

\newcommand{\Frac}[2]{\displaystyle{\frac{\displaystyle{#1}}{\displaystyle{#2}}}}

\newcommand{\beea}{\begin{eqnarray}}
\newcommand{\eeea}{\end{eqnarray}}

\newcommand{\bfe}{{\mbox{\boldmath $e$}} }

\newcommand{\bfz}{{\mbox{\boldmath $z$}} }
\newcommand{\0}{{\mbox{\boldmath $0$}} }

\newcommand{\BF}{\begin{footnotesize}}
\newcommand{\EF}{\end{footnotesize}}
\newcommand{\bi}{\begin{itemize}}
\newcommand{\ei}{\end{itemize}}
\newcommand{\ed}{\end{document}}
\newcommand{\be}{\begin{equation}}
\newcommand{\ba}{\begin{array}}
\newcommand{\ea}{\end{array}}
\newcommand{\ee}{\end{equation}}
\newcommand{\eeq}[1]{\label{eq:#1}\end{equation}}
\newcommand{\real}{{\mathbb R}}

\newcommand{\nat}{{\mathbb N}}

\newcommand{\bfx}{\mbox{\boldmath $x$}}
\newcommand{\bfy}{\mbox{\boldmath $y$}}

\newcommand{\bfxi}{\mbox{\boldmath $\xi$}}

\newcommand{\bfphi}{\mbox{\boldmath $\varphi$}}

\newcommand{\bfv}{{\mbox{\boldmath $v$}} }
\newcommand{\bfu}{{\mbox{\boldmath $u$}} }

\newcommand{\bfw}{{\mbox{\boldmath $w$}} }
\newcommand{\bff}{{\mbox{\boldmath $f$}} }
\newcommand{\bfa}{{\mbox{\boldmath $a$}} }

\newcommand{\bfR}{{\mbox{\boldmath $R$}} }

\newcommand{\bfN}{{\mbox{\boldmath $N$}} }
\newcommand{\bfh}{{\mbox{\boldmath $h$}} }

\newcommand{\cala}{{\cal A}}

\newcommand{\cald}{{\cal D}}

\newcommand{\calf}{{\cal F}}

\newcommand{\cali}{{\cal I}}

\newcommand{\bfeta}{\mbox{\boldmath $\eta$}}

\newcommand{\bfV}{{\mbox{\boldmath $V$}} }
\newcommand{\bfU}{{\mbox{\boldmath $U$}} }

\newcommand{\bfb}{{\mbox{\boldmath $b$}} }

\newcommand{\bfn}{{\mbox{\boldmath $n$}} }

\newcommand{\half}{\mbox{$\frac{1}{2}$}}

\def\Bbb R{\real}
\def\hat{\widehat}
\def\tilde{\widetilde}
\def\bar{\overline}

\newcommand{\bfchi}{\mbox{\boldmath $\chi$}}

\newcommand{\bfdelta}{\mbox{\boldmath $\delta$}}

\newcommand{\ED}{\end{description}}

\newcommand{\dist}{\mbox{\rm dist\,}}

\newtheorem{definition}{Definition}[section]
\newcommand{\Bd}{\begin{definition}\begin{rm}}
\newcommand{\Ed}{\end{rm}\end{definition}}
\newtheorem{remark}{Remark}[section]
\newcommand{\Br}{\begin{remark}\begin{rm}}
\newcommand{\Er}{\end{rm}\end{remark}}
\newtheorem{proposition}{Proposition}[section]
\newcommand{\Bp}{\begin{proposition}\begin{sl}}
\newcommand{\EP}[1]{\end{sl}\label{proposition:#1}\end{proposition}}
\newcommand{\propref}[1]{{\rm Proposition \ref{proposition:#1}}}
\newcommand{\Bt}{\begin{theorem}\begin{sl}}
\newcommand{\Et}{\end{sl}\end{theorem}}
\newcommand{\Bl}{\begin{lemma}\begin{sl}}
\newcommand{\El}{\end{sl}\end{lemma}}
\newtheorem{theorem}{Theorem}[section]
\newtheorem{lemma}{Lemma}[section]

\newtheorem{corollary}{Corollary}[section]
\renewcommand{\eqref}[1]{{\rm (\ref{eq:#1})}}
\newcommand{\Bc}{\begin{corollary}\begin{sl}}
\newcommand{\Ec}{\end{sl}\end{corollary}}
\newcommand{\ET}[1]{\end{sl}\label{theorem:#1}\end{theorem}}
\newcommand{\EDD}[1]{\end{rm}\label{definition:#1}\end{definition}}
\newcommand{\EL}[1]{\end{sl}\label{lemma:#1}\end{lemma}}

\newcommand{\ER}[1]{\end{rm}\label{remark:#1}\end{remark}}
\newcommand{\EC}[1]{\end{sl}\label{corollary:#1}\end{corollary}}
\newcommand{\defref}[1]{{\rm Definition \ref{definition:#1}}}

\newcommand{\lemmref}[1]{{\rm Lemma \ref{lemma:#1}}}

\numberwithin{equation}{section}
\usepackage{enumerate}
\usepackage[notcite,notref,color]{showkeys}
\begin{document}

\title{Equilibrium Configurations and their Uniqueness\\ in a Fluid-Solid Interaction Problem} 
\author{Denis Bonheure\thanks{D\'epartement de Math\'ematique, Universit\'e Libre de Bruxelles, Belgium}\,,
\  Giovanni P. Galdi\thanks{Department of Mechanical Engineering and Materials Science, University of Pittsburgh, USA.}
\  \&\,\ Clara Patriarca\thanks{D\'epartement de Math\'ematique, Universit\'e Libre de Bruxelles, Belgium} }
\date{}
\maketitle\noindent

\begin{abstract}
	\noindent
We demonstrate existence in the ``large" and uniqueness in the ``small" of equilibrium configurations for the coupled system consisting of a Navier-Stokes fluid interacting with a rigid body subjected to spring forces and restoring moments. The driving mechanism is a uniform, given velocity field of the fluid at large spatial distances from the body. The main difficulty in the proof of the above properties arises from the fact that the body can rotate around a given axis, which produces a highly nonlinear problem. 
	
	\bigskip
	\noindent
	\textbf{AMS subject classification}: 74F10, 76-10, 76D05, 35Q30 \\
	\textbf{Keywords}: Fluid-structure interaction, Equilibrium, Rotation, Navier-Stokes equations.  
\end{abstract} 

\color{black}
\section{Introduction}
In this article, we carry out a rigorous study of the equilibrium configurations of a rigid body, $\mathscr B$, subjected to restoring elastic forces and torques, interacting with a viscous fluid driven by a time-independent uniform flow at large distance from $\mathscr B$,  characterized by a prescribed  vector $\bfV$. Our goal is precisely  to investigate the problems of their existence and  uniqueness.
\par
The motivation for this type of analysis stems from, but is not limited to, the investigation of the structural properties of suspension bridges interacting with wind flow which, as is known, is one of the most relevant engineering problems in the general field of flow-induced oscillations; see, e.g., the monographs \cite{Bl,NaRo}.
\par
The model we adopt here is more general than the one usually chosen by engineers \cite{Bl,NaRo} and that we considered in similar investigations \cite{BeBoGaGaPe,BoGaGa1,BoGaGa2}. Actually, unlike all these works, we allow $\mathscr B$ also to rotate along a given direction while subjected to a restoring torque. This way we are able to provide more accurate description of suspension bridges where the deck may undergo angular displacement, $\theta$, around its longitudinal direction. It is known that the torsional movement often couples with vertical oscillations and this coupling amplifies stresses and can destabilize the entire structure. Such a model, while representing a more realistic physical situation, at the same time introduces several mathematical difficulties that make the study particularly challenging. To explain this point, we observe that the generic configuration of $\mathscr B$ (compact set of $\real^3$) is characterized by the position vector, $\bfdelta$, of the center of mass of $\mathscr B$  and the angle $\theta$ around the axis of rotation, both evaluated from the unloaded position $(\bfdelta\equiv\0,\theta=0)$ of the elastic force and torque. Denote by $\mathscr S_0$   the reference configuration  of the coupled system body-liquid, where $\mathscr B$ is in its unloaded configuration and the liquid is at rest in the region $\Omega:=\real^3\backslash\mathscr B$, corresponding to $\bfV=\0$.  Then, for a given $\bfV\in \real^3\backslash\{\0\}$, determining an equilibrium configuration for the coupled system body-liquid,  means to find $(\bfdelta_0,\theta_0)\in \real^3\times\real$,  such that the liquid flow regime is independent of time (steady state). Thus, since the region occupied by the liquid also depends on $\theta_0$,  the search of equilibria  becomes  rather intricate. To address this problem, we rewrite the equations of motion in a system of coordinates attached to the reference configuration, so that the region of flow becomes fixed and known. However, in doing so, the direction of the vector $\bfV$ becomes then a  nonlinear function of $\theta$, which  furnishes the coupling between the motion of the liquid and the (nonlocal) equation of balance of the angular momentum for $\mathscr B$; see \eqref{9_0}-\eqref{8}. 
\par
For this reformulated problem, our first goal is to prove the existence of solutions without imposing restrictions on the magnitude of $\bfV$. An appropriate tool for this purpose is the corollary of Leray-Schauder degree theory provided by the Schaefer fixed point theorem \cite{Sh}. However, this theorem requires the compactness property of the associated  map, which, in our case, is not available since the fluid motion occurs in the exterior domain $\Omega$. To overcome this issue, we apply the ``invading domains" approach, that is, we consider the given problem on each element of an increasing sequence of bounded domains $\{\Omega_k\}$ whose union coincides with $\Omega$, endowed with homogeneous boundary conditions on the ``fictitious" boundary; see \eqref{24}. We then show that, in each $\Omega_k$, the problem can be formulated as a nonlinear equation in a suitable Banach space (see \eqref{28}), and that the relevant (nonlinear) map is compact (see Lemma \ref{lemma:com}). Moreover, we prove that, for each $k\in\nat$, all possible solutions satisfy a uniform bound that depends only on $\Omega$,  and on the nondimensional number $\lambda$ related to $V=|\bfV|$; see  \eqref{lamb}. Therefore, by Schaefer theorem, we establish the existence of a solution ${\sf s}_k$ in each $\Omega_k$ (see Proposition \ref{proposition:R}) and, thanks to the uniform bound, from the sequence $\{{\sf s}_k\}$ we can extract a subsequence that converges to a weak solution of the original problem in $\Omega$; see Proposition \ref{prop:1}. Indeed, by classical results, it follows that velocity and pressure fields associated to such a solution, on the one hand, are of class $C^\infty(\Omega)$ and, on the other hand, satisfy sharp summability properties ``far" from $\mathscr B$; see Theorem \ref{th:1}. We next address the problem uniqueness of equilibria, that is, uniqueness of solutions given in Theorem \ref{th:1} in their own class of existence; see Theorem \ref{th:unicita}. Because of the highly nonlinear character of the equations, the proof of such a property cannot be carried out by classical energy arguments. Indeed,  we employ a more general strategy that develops according to the following two steps. In the first one, we prove that, for all values of $\lambda$ below a certain constant depending only on $\Omega$, the corresponding solutions  can be bounded from above, in suitable norms, by an affine function of $\lambda$; see Lemma \ref{Lq-estimate}. This result is, in turn, obtained by employing, on the fluid equations, appropriate estimates in {\em homogeneous} Sobolev spaces for the Oseen problem in exterior domain  combined with {a successive approximation argument}. With the properties shown in Lemma \ref{Lq-estimate} in hand, we can then prove in Theorem \ref{th:unicita} that, if $\lambda$ is below a constant depending only on $\Omega$, the solution constructed in Theorem \ref{th:1} is necessarily unique in its class of existence.       
\par

{At this stage, it is worth mentioning that in the present work we do not tackle the question of the stability of the unique equilibrium configuration at low Reynolds in the present work. A stability analysis for a similar problem, namely in the case of a 2D infinite channel where the dynamics is driven by a Poiseuille velocity field at large distance of the solid, has been performed in \cite{BoHiPaSpe} while for the model \eqref{01} \textit{without} rotation it has been treated in \cite{BoGaGa1}. The presence of the rotation makes the method of \cite{BoGaGa1} inapplicable and requires a new approach which will be the subject of subsequent investigations.} 

The outline of the paper is the following. In Section 2 we provide a precise formulation of the problem and rewrite the relevant equations in a frame attached to the reference configuration. In Section 3, we define the equilibrium solutions and prove their existence by the method described above. Finally, Section 4 is dedicated to the proof of the uniqueness of equilibria.

\section{Formulation of the Problem and Functional Setting}

Consider a rigid body, $\mathscr B$, occupying the closure of the bounded domain $\Omega_0\subset\mathbb{R}^3$, completely immersed in a Navier-Stokes liquid, $\mathscr L$, that fills the entire three-dimensional space outside $\mathscr B$. We shall assume that $\Omega_0$ is of class $C^2$, even though some results, like \propref{1}, continue to hold without any regularity assumption. The flow of the liquid becomes equal to the velocity field $-V\tilde{\bfb}_\alpha$ at large distances from $\mathscr B$, where $V$ is a prescribed real number that, without loss, we assume to be positive, and $\tilde{\bfb}_\alpha\in S^2$, with $S^2\subset \mathbb{R}^3$  the unit sphere in $\mathbb{R}^3$. Precisely, 
  $$\tilde{\bfb}_\alpha=\cos\alpha\, \tilde\bfe_1 +\sin\alpha\, \tilde\bfb \quad \text{with} \quad \tilde\bfb = (0,\tilde b_2,\tilde b_3)\,,
  $$
 where the (given) angle $\alpha\in\mathbb{R}$ can be seen as an angle of attack, and $\tilde \bfb\in\mathbb{R}^3$ has lenght 1.  When $\alpha=0$, the flow is aligned with the axis of rotation of the body. At $90^\circ$, the flow is totally transverse. 
\par We suppose that, with respect to an inertial frame $\cali:=\{O',\bfe_1,\bfe_2',\bfe_3'\}$, $\mathscr B$  moves subject to the following force and torque: (a) a linear, possibly anisotropic, restoring force $\bfR=-{\mathbb A}\cdot\bfchi$ with ${\mathbb A}$ symmetric, positive definite matrix ({\em stiffness matrix}) and $\bfchi:=\bfy_G-\bfy_{O'}$ displacement of the center of mass $G$ of $\mathscr B$ with respect to  $O'$; (b) a restoring torque $-{k}\theta\bfe_1$ with ${k}>0$ (\textit{stiffness constant}) and $\theta$ angle counted with respect to the direction $\bfe_2'$ \emph{with its multiplicity} so that the restoring torque takes into account the number of full turns made by the body. We will assume that $\mathscr B$ is free to rotate only around  $\bfe_1$. 
 \par Let $\rho, \nu$ and $I$ be density and kinematic viscosity of the liquid and moment of inertia of $\mathscr B$ around $\tilde{\bfe}_1$, respectively. We scale  velocities by $V$, length by ${\sf d}:=(\rho/I)^{\frac15}$ and time by $d^2/\nu$. One can then show that the equations of motion of the coupled system body-liquid $\mathscr{S}$ when referred to $\cali$ are given in non-dimensional form by  \cite[(2.1)]{BoGaPa}
\be\left\{\ba{l}\medskip\left.\ba{rr}\medskip
\partial_t{\bfw}+\lambda{\bfw}\cdot\nabla{\bfw}&=\Delta{\bfw}-\nabla {{\sf p}}\\
& \Div{\bfw}=0\ea\right\}\ \ \mbox{in $\BCup{t\in(0,\infty)}{}\left(\calf(t)\times\{t\}\right)$}\,,\\ \medskip
\ \ {\bfw}(\bfy,t)={\bfeta}(t)+\lambda^{-1}\omega(t)\bfe_1\times({\bfy}-{\bfchi}(t))\,, \ \mbox{at $\BCup{t\in(0,\infty)}{}\left(\partial\calf(t)\times\{t\}\right)$}\,;\ \\ \medskip
\ \ \Lim{|{\bfy}|\to\infty}{\bfw}(\bfy,t)=- \tilde\bfb_\alpha\,,\ t\in(0,\infty)\,,\\
\medskip\left.\ba{lr}\medskip
\dot{{\bfeta}}+{\mathbb A}\cdot{\bfchi}+\mu\Int{\partial\calf(t)}{} \mathbb{T}({\bfw},{\sf p})\cdot\bfN=\0, \qquad
\\ \medskip
\dot{{\bfchi}}={{\bfeta}},\\ \medskip
\dot{\omega}+k\theta+\lambda\bfe_1\cdot\Int{\partial\calf(t)}{}({\bfy}-\bfchi)\times \mathbb{T}({\bfw},{\sf p})\cdot\bfN=0,\, \\
\medskip
\dot{\theta}={\omega},
\ea\right\}\ \ \mbox{in $(0,\infty)$,} \qquad 
\ea\right.
\eeq{01}
where  ${\bfw}$ and ${\sf p}$ are velocity and pressure fields of the liquid, 
$\mu:=\rho\, {\sf d}^3/M$, with $M$ mass of $\mathscr B$, and
\be 
\lambda:=\frac{V {\sf d}}\nu\
\eeq{lamb}
is the relevant Reynolds number in this model. 
Of course, all the quantities involved in \eqref{01} are understood to be non-dimensional.

The fluid domain $\calf(t)$ is
the domain occupied by $\mathscr L$ at time $t\ge0$. It depends on the position of the center of mass of the body and its orientation, i.e. 
\be
\calf(t)= \calf(\bfchi(t),\theta(t))=\left(\bfy_{ O'} + {\bfchi}(t) + \mathbb{Q}(\theta(t))\Omega_0\right)^\complement\,,
\eeq{F}
where $\mathbb Q(t)=\mathbb{Q}(\theta(t))$,  is the one-parameter
family of rotations around  $\bfe_1$ defined by 
\be
\mathbb{Q}(\theta(t)):=\left(\ba{ccc}\medskip 1 &0&0\\ \medskip
0&\cos\theta(t)&-\sin\theta(t)\\
0&\sin\theta(t)&\cos\theta(t)\ea\right)\,.
\eeq{Q}   
We also recall the standard notation $\mathbb T(\cdot,\cdot)$ that denotes the Cauchy stress tensor, that is
$$
\mathbb T(\bfz,\psi):=2\,\mathbb D(\bfz)-\psi\,\mathbb I\,,\ \ \ \mathbb D(\bfz):=\frac12\big(\nabla\bfz+(\nabla\bfz)^\top\big)\,\qquad  \forall\,(\bfz,\psi) \in \mathbb{R}^3\times\mathbb{R}\,,
$$
with $\mathbb I$ identity matrix, and  $\bfN$ the unit outer normal at $\calf$ directed toward $\mathscr{B}$.
\par 
Our aim is to study the existence and uniqueness of equilibria of \eqref{01}, where, by ``equilibrium" we mean a time-independent solution to \eqref{01}. 
\par
Instead of working with system \eqref{01}, which is set on a time-dependent fluid domain, we rather reformulate \eqref{01} into an equivalent system using a body-fixed frame. In this configuration, we denote by $\Omega$ the space filled by the liquid $\mathscr L$ for any possible position of the body, once we consider the origin of coordinates in the interior of the body, that is we take as a reference system $\mathcal{I}_G=\{G, \bfe_1,\bfe_2,\bfe_3\}$, 
where 
$$
(\bfe_1,\bfe_2,\bfe_3)=\mathbb{Q}^\top(\theta)\cdot(\bfe_1,\bfe'_2,\bfe'_3)\,.
$$
This clearly implies that any position vector $\bfx$ of a point $P$ in $\mathcal{I}_G$ is related to the same point $\bfy$ in $\mathcal{I}$ by the following relation:
$$
\bfy=\mathbb{Q}(\theta)\cdot \bfx+\bfchi\,.
$$ 
Given  $\bfn$ the unit outer normal at $\partial\Omega$ directed toward $\mathscr{B}$, proceeding as in \cite[\S\S 1 and 2.1]{Gah}
we then deduce  that $({\bfw}, {\sf p},\bfchi,\theta)$ solves \eqref{01} if and only if $(\bfv,{ p},\bfdelta,\theta)$ satisfies the following set of equations 
\be\left\{\ba{l}\medskip\left.\ba{rr}\medskip
\partial_t{\bfv}+\lambda({\bfv}-{\bsfV })\cdot\nabla{\bfv}+\dot{\theta}{\bfv}^\perp & =\Delta{\bfv}-\nabla { p}\\
& \Div{\bfu}  =0\ea\right\}\ \ \mbox{in $\Omega\times(0,\infty)$}\,,\\ \medskip
\ \ {\bfv}(\bfx,t)=\bsfV(\bfx,t):=\bfxi(t)+\lambda^{-1}{\omega}(t)\bfx^\perp\,, \ \mbox{at $\partial\Omega_0\times(0,\infty)$}\,;\ \\ \medskip
\ \ \Lim{|\bfx|\to\infty}{\bfv}(\bfx,t)=\bfb_\alpha(\theta)\,,\ t\in(0,\infty)\,,\\
\medskip\left.\ba{lr}\medskip
\dot{\bfxi}+\dot{\theta}\bfxi^\perp+{\mathbb B}(\theta)\cdot	\bfdelta+\mu\Int{\partial\Omega}{} {\mathbb T}({\bfv},{ p})\cdot\bfn=\0
\\ \medskip
\dot{\bfdelta}+\dot{\theta}\bfdelta^\perp={\bfxi}\\ \medskip
\dot{\omega}+k\theta+\lambda\,\bfe_1\cdot\Int{\partial\Omega}{}\bfx\times {\mathbb T}(\bfv,{ p})\cdot\bfn=0\, \\
\medskip
\dot{\theta}={\omega}
\ea\right\}\ \ \mbox{in $(0,\infty)$\,.}
\ea\right.
\eeq{01_nrs}
with 
\be
\bfb_\alpha(\theta):= \mathbb Q^\top(\theta)\cdot\tilde\bfb_{\alpha}\,, \qquad \mathbb B(\theta):=\mathbb Q^\top(\theta)\cdot\mathbb A\cdot\mathbb Q(\theta)\,.
\eeq{9_0}
It is clear that there exist $\rho_1,{\rho}_2>0$, depending on $\mathbb A$, such that 
\be
\rho_1|\bfdelta|^2\le \bfdelta\cdot\mathbb{B}\cdot \bfdelta\le {\rho}_2\,|\bfdelta|^2\,.
\eeq{bound_B}

\par To conclude this section, we briefly recall the relevant functional spaces that will be used throughout the paper. Given a domain $A\subset \mathbb{R}^3$, $1\le q\le\infty, m\in \mathbb{N}$, $W^{m,q}(A)$ denotes the usual Sobolev space with norm $\|\cdot\|_{m,q,A}$, and its subspace $W^{m,q}_0(A)$ given by the completion of $C^\infty_0(A)$ in the same norm. Clearly $L^q(A)=W^{0,q}(A)$ is the Lebesgue space with norm $\|\cdot\|_{q,A}$. 
Since problem \eqref{01_nrs} is set in an exterior domain, of particular importance will be the homogeneous Sobolev space $D^{m,q}=D^{m,q}(A)$, with semi-norm $\sum_{|l|=m}\|D^l u\|_{q,A}$.
Whenever confusion of domains does not occur, we will omit the corresponding subscript in the previous norms. Likewise for the $L^2$-scalar product $(\cdot,\cdot)=(\cdot,\cdot)_A$.
\par{Besides standard spaces, we shall use certain function classes that are characteristic of hydrodynamics problems related to a viscous incompressible fluids. In particular, we define
$$\cald(A):=\{\bfphi\in C_0^\infty(A):\Div\bfphi=0 \,\,\text{in}\,\, A\}\,.
$$
Then, $\cald_0^{1,2}(A)$ denotes completion of $\cald(A)$ in the $\|\nabla(\cdot)\|_{2,A}$-norm; moreover, $[\cdot,\cdot]_{A}:=\int_\Omega\nabla\bfu_1:\nabla\bfu_2$ is the scalar product in $\cald_0^{1,2}(A)$. If $\cala=\real^3$ or $\cala\subset\real^3$ is a bounded domain with $\bar{\Omega_0}\subset \cala$, we set
$$
\begin{aligned}
\mathcal{C}(\cala):=\{\bfphi\in C_0^\infty(\cala)\,:\, &\Div\bfphi=0 \,\,\text{in}\,\, \cala\,\quad \\& \text{and} \quad \exists\,(\hat{\varphi},\hat{\beta})\in\mathbb{R}^3\times\mathbb{R}\,\,\text{s.t.}\, \bfphi(\bfx)=\hat{\varphi}+\hat{\beta}\bfx^\perp\,\text{in a neighborhood of $\Omega_0$}\}\,.
\end{aligned}
$$
and define
\be
\mathcal{K}(\cala)\,:=\ \left\{\mbox{completion of $\mathcal C(\cala)$  in the $\|\nabla(\cdot)\|_{2}$-norm}\right\}\,.
\eeq{Ka}
}
{More details on the functional framework can be retrieved from \cite{Ana-Paolo_Lady}.}
\section{Existence of Equilibria}\label{sec:equilibria}
{We look for equilibria of the coupled system $\mathscr S$}. These states occur when (and only when) the liquid is in a steady regime and the structure occupies a time-independent configuration.  
Mathematically, from \eqref{01_nrs}, finding steady states means that we 
must find functions $(\bfv,p,\bfdelta,\theta)$  such that, for all $\bfx\in\mathbb{R}^3$,
$$
\bfv=\bfv(\bfx)\in\mathbb{R}^3\,,\, p=p(\bfx)\in \mathbb{R}\,,\, \bfdelta\in \real^3\,,\ \theta \in\real\,,
$$
satisfy the following set of (dimensionless) equations 
\be\left\{\ba{lc}\medskip\left.\ba{ll}\medskip
{\lambda}\,\bfv\cdot\nabla\bfv=\Delta\bfv-\nabla p\\
\qquad\qquad\quad\Div\bfv=0\ea\right\}\ \ \mbox{in $\Omega$}\\
\bfv(\bfx)=\0 \ \mbox{at $\partial\Omega$}\,;\ \ \Lim{|\bfx|\to\infty}\bfv(\bfx)=-\bfb_\alpha(\theta)\,,
\ea\right.
\eeq{7}
\be\left\{\ba{ll}\medskip
\mathbb B(\theta)\cdot\bfdelta+\mu\Int{\partial\Omega}{}\mathbb T(\bfv,p)\cdot\bfn=\0\\
k\theta+\lambda\,\bfe_1\cdot\Int{\partial\Omega}{}\bfx\times\mathbb T(\bfv,p)\cdot\bfn=0\,.
\ea\right.
\eeq{8}

The main goal of the initial part of this paper is to show that, for any $\lambda>0$ problem \eqref{7}--\eqref{8} admits at least one smooth solution, satisfying suitable summability conditions at ``large" spatial distances. This is exactly the content the following theorem.
\begin{theorem}\label{th:1}
	For any $\lambda>0$, there is at least one solution $(\bfv,p,\bfdelta,\theta)\in C^\infty(\Omega)\times C^\infty(\Omega)\times\real^3\times\real$ to \eqref{7}--\eqref{8} that, in addition, satisfies 
	\be
	(\bfv+\bfb_\alpha(\theta),p) \in [L^q(\Omega)\cap D^{1,r}(\Omega)\cap D^{2,s}(\Omega)]\times[L^{\sigma}(\Omega)\cap D^{1,s}(\Omega)]\,,
	\eeq{REG}
	for all $q\in (2,\infty]$, $r\in(\frac43,\infty]$, $s\in (\frac32,\infty]$, $\sigma\in (1,\infty)$. 
\end{theorem}

\par   
The proof of Theorem \ref{th:1} is posponed to the end of this section. {It will be an immediate consequence of an auxiliary result, contained in Proposition \ref{prop:1}. In order to state it, }  we need some preparatory work.
  
To this end, we begin to observe that,  to show existence for the full problem \eqref{7}--\eqref{8}, it is enough to prove the same property for the reduced problem \eqref{7}, \eqref{8}$_2$,  since, once $(\bfv,p,\theta)$ is found, then $\bfdelta$ is directly given by
\be
\bfdelta:=-\mu\,\mathbb  B^{-1}(\theta)\cdot\Int{\partial\Omega}{}\mathbb T(\bfv,p)\cdot\bfn\,,
\eeq{delta}
provided $(\bfv,p)$ has enough regularity as to guarantee that the surface integral  be well-defined. 
The same simplification arises if $\alpha=0$ (or $\alpha=\pi$). Indeed, then $\bfb_\alpha=\bfe_1$ ($\bfb_\alpha=-\bfe_1$) is $\theta$-independent and 
once we have $(\bfv,p)$ solving \eqref{7}, then $\theta$ is given by
$$
\theta = -\frac\lambda k\,\bfe_1\cdot\Int{\partial\Omega}{}\bfx\times\mathbb T(\bfv,p)\cdot\bfn\,,
$$
provided again that $(\bfv,p)$ has enough regularity as to guarantee that the surface integral  be well-defined. 
With this in mind, we assume $\alpha\ne 0\ ({\rm mod}\, \pi)$. From this point onward, we will suppress the explicit dependence on $\theta$ in $\bfb_\alpha$, unless  strictly necessary. 
\par
Thus, existence of equilibria is reduced to prove existence for the following system of equations
\be\ba{cc}\medskip\left.\ba{ll}\medskip
\lambda(\bsfu-\bfb_\alpha)\cdot\nabla\bsfu=\Delta\bsfu-\nabla p\\\qquad \qquad \qquad\quad\quad
\Div\bsfu=0\ea\right\}\ \ \mbox{in $\Omega$}\\
\bsfu(\bfx)=\bfb_\alpha\ \mbox{at $\partial\Omega$}\,;\ \ \Lim{|\bfx|\to\infty}\bsfu(\bfx)=\0
\ea
\eeq{10}
\be\ba{ll}\medskip
 \theta+\hat{\lambda}\,\bfe_1\cdot\Int{\partial\Omega}{}\bfx\times\mathbb T(\bsfu,p)\cdot\bfn=0\,,
\ea
\eeq{11}
with
$$
\bsfu:=\bfv+\bfb_\alpha\,,\ \ \hat{\lambda}:=\frac\lambda k\,.
$$
\begin{remark}
As mentioned previously, the angle $\theta\in \real$ is counted with its multiplicity, which describes the number of complete rotations, $n\in \mathbb Z$, around $\bfe_1$ necessary for the body to be in the equilibrium configuration corresponding to the given $\lambda$. Thus, $\theta=\theta_1+2n\pi$, with $\theta_1\in [0,2\pi)$. {One can prove that $\theta$ is small if $\lambda$ is sufficiently ``small" (so that $n=0$ in that case) but we cannot exclude that $n\ne 0$ when $\lambda$ is large. }
\end{remark}
\par
Objective of this section is to show that problem \eqref{10}--\eqref{11} admits at least one weak solution (suitably defined), for arbitrary values of $\lambda$. 
In this regard,
we begin to prove a general result concerning the lift of boundary data, in the sense provided by the next lemma. This proof is along the same lines of that given in \cite[Lemma 2.3]{BoGa}, but with difference in details, therefore we report it here for the sake of completeness. 

To this purpose, we introduce, for any $R>R_*=:\text{diam}(\Omega_0)$, the domains 
$$
\Omega_R:=\Omega\cap B_R\,,
$$
with $B_r:=\{\bfx\in\mathbb{R}^3\,:\, |\bfx|<r\,,r>0\}$. Then, the following result holds.

\Bl There exists a linear, continuous map 
$$
\bfU:\bfa\in S^2\mapsto \bfU(\bfa)\in W^{2,2}(\Omega)
$$  
with the following properties:
\begin{itemize}
   \item [{\rm (i)}] There is $\rho_0>R_*$, such that   $\supp(\bfU(\bfa))\subset \Omega_{\rho_0}$, for all $\bfa$;
 \item [{\rm (ii)}] $\bfU=\bfa$\,,\ \  in a neighborhood of $\partial\Omega$; 
\item [{\rm (iii)}] $\Div\bfU=0\ \ \mbox{in}\ \Omega$;
 \item [{\rm (iv)}]  $\|\bfU\|_{m,q}\le \gamma$, for all $(m,q)\in \nat\times [1,\infty]$, with $\gamma>0$ depending only on $\Omega,m,q$; 
  \item [{\rm (vi)}] For any $\bfw \in W^{1,2}(\Omega)$, $\bfz\in \cald_0^{1,2}({\Omega})$,

$$
\int_{{\Omega}}\left|\bfw\cdot\nabla\bfz\cdot\bfU\right|\le \half\,\|\nabla\bfw\|_2\|\nabla\bfz\|_2\,.
$$
\item [{\rm (vii)}] There exists $c>0$, depending on $\Omega$, such that 
$$
\|\bfU(\bfa_1)-\bfU(\bfa_2)\|_{2,2}\le c|\bfa_1-\bfa_2|\,,\ \ \forall\,\bfa_1,\bfa_2\in S^2\,.
$$
\end{itemize}
\label{lemma:lifting}
\EL{1} 
\begin{proof} Let $\psi:r\in (0,\infty)\mapsto [0,\infty)$ be a smooth, non-decreasing  function such that $\psi(r)=0$ if $r\le1$, and $\psi(r)=1$, if $r\ge2$, and, for a given $\varepsilon>0$ (to be fixed later), set
$$
\phi(\varepsilon;\bfx)=\psi(-\varepsilon \ln d(\bfx)),
$$
where $d(\bfx):=\dist(\bfx,\partial\Omega)$.
The following properties are easily checked:
$$
\phi(\varepsilon;\bfx)=\left\{\ba{ll}\medskip 1 & \mbox{if}\ d(\bfx)\le {\rm e}^{-2/\varepsilon}\\
0 & \mbox{if}\ d(\bfx)\ge {\rm e}^{-1/\varepsilon}
\ea\right.\,,$$
\be\ba{cc}\medskip
 \supp(\phi)\subset  \{\bfx\in\Omega: \,d(\bfx)\le {\rm e}^{-1/\varepsilon}\}=:\Omega_\varepsilon\,,\\ \supp(\nabla\phi)\subset \{\bfx\in\Omega: \,{\rm e}^{-2/\varepsilon}\le d(\bfx)\le {\rm e}^{-1/\varepsilon}\}\,,
\ea
\eeq{suP}
and, moreover,
\be
|\nabla\phi(\varepsilon;\bfx)|\le \frac{c\,\varepsilon}{d(\bfx)}\,,
\eeq{1.2}
with $c$ independent of $\varepsilon$. For a given $\bfa\in S^2$, any $\bfx=(x_1,x_2,x_3)\in\mathbb{R}^3$, let
\be
\bsfU(\bfx;\bfa):=x_3a_2\bfe_1+x_1a_3\bfe_2+x_2a_1\bfe_3\,,
\eeq{U}
and define
$$
\bfU(\bfx;\bfa):=\curl(\phi(\varepsilon;\bfx)\bsfU(\bfx;\bfa))\,,
$$
from which the validity of (iii) as well as the linearity property  of the map $\bfa\mapsto \bfU(\bfa)$  follow at once.
Since $\curl\bsfU=\bfa$, we get
\be
\bfU(\bfx;\bfa)=\phi(\varepsilon;\bfx)\bfa-\bsfU(\bfx;\bfa)\times\nabla\phi(\varepsilon;\bfx)\,.
\eeq{1.3}
Thus, {we infer from \eqref{suP}-\eqref{1.3} that}
\be
\supp(\bfU)\subset \Omega_\varepsilon\,,\ \ \bfU=\bfa\ \,\mbox{in $\Omega_{\varepsilon/2}$}\,,
\eeq{SuP}
and also
\be
\|\bfU\|_{2,2}\le \gamma_\varepsilon \,;\ \ \|\bfU(\bfa_1)-\bfU(\bfa_2)\|_{2,2}\le \gamma_\varepsilon|\bfa_1-\bfa_2|\,,\ \ \bfa_1,\bfa_2\in S^2\,,
\eeq{1.3_0}
{where $\gamma_\varepsilon$ is positive constant } depending only on $\Omega$ and $\varepsilon$. 
Next, by Schwarz inequality,  \eqref{1.2}-\eqref{1.3}, and observing that $\sup_{\Omega_\varepsilon}|\bsfU|\le c$,
we deduce  
\be\ba{rl}\medskip
\left(\Int{{\Omega}}{}\left|\bfw\cdot\nabla\bfz\cdot\bfU\right|\right)^2 \le \|\nabla\bfz\|_2^2\Int{{\Omega_\varepsilon}}{}|\bfU|^2|\bfw|^2
& \le c\,\|\nabla\bfz\|_2^2\Int{{\Omega_\varepsilon}}{}\left(\phi^2|\bfw|^2+\varepsilon^2\Frac{|\bfw|^2}{d^2}\right)\\ \medskip
&\le c\,\|\nabla\bfz\|_2^2\Int{{\Omega_\varepsilon}}{}\left(d^2+\varepsilon^2\right)\Frac{|\bfw|^2}{d^2}\\ & \le c\,\varepsilon^2\,\|\nabla\bfz\|_2^2\Int{{\Omega_\varepsilon}}{}\Frac{|\bfw|^2}{d^2}\,,
\ea
\eeq{1.4} 
where $c$ is independent of $\varepsilon$.
{Using Hardy inequality, see e.g. \cite[Lemma III.6.3]{Gab}, in the last integral}, we conclude that 
$$
\left(\Int{{\Omega}}{}\left|\bfw\cdot\nabla\bfz\cdot\bfU\right|\right)^2 
\le c\,\varepsilon^2\,\|\nabla\bfz\|_2^2\Int{{\Omega_\varepsilon}}{}|\nabla\bfw|^2,
$$
which, after choosing $\varepsilon=1/(2\sqrt{c})$, proves the property (vi). Finally, (i),  (ii),  (iv) and (vii) as well as the continuity property are a consequence of this choice of $\varepsilon$ and \eqref{SuP}-\eqref{1.3_0}. 
\end{proof}
\par
Set 
$\bsfu:=\bfu+\bfU(\bfb_\alpha)$. Then, problem \eqref{10}--\eqref{11} becomes
\be\left\{\ba{ll}\medskip\left.\ba{rr}\medskip
\lambda\left((\bfu-\bfb_\alpha)\cdot\nabla\bfu+\bfU\cdot\nabla\bfu+\bfu\cdot\nabla\bfU\right) & =\Delta\bfu-\nabla p+\bff_\lambda(\bfb_\alpha) \\
&  \Div\bfu =0\ea\right\}\ \ \mbox{in $\Omega$}\\
\ \ \bfu(\bfx)=\0 \ \mbox{at $\partial\Omega$}\,;\ \ \Lim{|\bfx|\to\infty}\bfu(\bfx)=\0
\ea\right.
\eeq{12}
\be\ba{ll}\medskip
 \theta+\hat{\lambda}\,\bfe_1\cdot\Int{\partial\Omega}{}\bfx\times\mathbb T(\bfu,p)\cdot\bfn=0\,, & \hspace{5.2cm}\
\ea
\eeq{13}
where
\be
\bff_\lambda(\bfb_\alpha):=-\lambda(\bfU{-}\bfb_\alpha)\cdot\nabla\bfU+\Delta\bfU\,.
\eeq{14}
We shall next put \eqref{13} in a different (equivalent) form.
  Let $\varphi: r\in (0,\infty)\to [0,\infty)$  be a smooth cut-off such that $\varphi(r)=1$ if $r\le 1$, and $\varphi(r)=0$, if $r\ge 2$. For  a fixed $\rho>R_*$, set
\be
\bfh(\bfx):={\varphi}(|\bfx|/\rho)\bfe_1\times\bfx  = {\varphi}(|\bfx|/\rho)(0,-x_3,x_2)\,.
\eeq{h1}
After noticing that $\nabla{\varphi}(|\bfx|/\rho)\times\bfx=\0$, one easily shows that $\bfh$ satisfies the following properties 
\be
\bfh\in C^\infty(\Omega)\,;\ \ \Div\bfh=0\,\ \mbox{in}\,\Omega\,;\ \ \bfh(\bfx)\equiv\0\,\ \mbox{for}\, |\bfx|\ge 2\rho\,.
\eeq{h} 
If we dot-multiply both sides of \eqref{12}$_1$ by $\bfh$,  integrate by parts over $\Omega$ and use
 \eqref{h}, we get
\be
\begin{aligned}
\bfe_1\cdot\int_{\partial\Omega}\bfx\times\mathbb T(\bfu,p)\cdot\bfn&=\Int{\hat{\Omega}}{}\left\{\left[\lambda\left((\bfu-\bfb_\alpha(\theta))\cdot\nabla\bfu+{\bfU\cdot\nabla\bfu+\bfu\cdot\nabla\bfU}\right)-\bff_\lambda\right]\cdot\bfh+\mathbb D(\bfu):\mathbb D(\bfh)\right\}\\&=:-F_\lambda(\bfu;\bfb_\alpha)\,,
\end{aligned}
\eeq{16}
where $\hat{\Omega}$ is a fixed {\em bounded} domain containing $\supp(\bfh)$.
As a result, problem \eqref{12}--\eqref{13} can be restated as follows:
\be\left\{\ba{lc}\medskip\left.\ba{lr}\medskip
\lambda\left((\bfu{-}\bfb_\alpha)\cdot\nabla\bfu+\bfU\cdot\nabla\bfu+\bfu\cdot\nabla\bfU\right)&=\Delta\bfu-\nabla p+\bff_\lambda(\bfb_\alpha)\\
&\Div\bfu=0\ea\right\}\ \ \mbox{in $\Omega$}\\ \medskip
\ \ \bfu(\bfx)=\0 \ \mbox{at $\partial\Omega$}\,;\ \ \Lim{|\bfx|\to\infty}\bfu(\bfx)=\0\,;\\
 \ \ \theta=\hat{\lambda}\,F_\lambda(\bfu;\bfb_\alpha)\,,
\ea\right.
\eeq{17}
with  $\bfb_\alpha$ given in \eqref{9_0}, $\bfU=\bfU(\bfb_\alpha)$, $\bfh$  in \lemmref{1} and \eqref{h}, and $\bff_\lambda$, $F_\lambda$  in \eqref{14} and \eqref{16}.\par
\smallskip\par

For problem \eqref{17} we give the following definition of weak solution.
\Bd The pair $(\bfu,\theta)$ is a {\em weak solution} to \eqref{17} if
\begin{itemize}
\item[{\rm (i)}] $(\bfu,\theta)\in \cald_0^{1,2}(\Omega)\times\real$\,;
\item[{\rm (ii)}] $(\bfu,\theta)$ satisfies
\be\left\{\ba{rl}\medskip
\left[\bfu,\bfphi\right]&=-\left(\lambda\left((\bfu{-}\bfb_\alpha)\cdot\nabla\bfu+\bfU\cdot\nabla\bfu+\bfu\cdot\nabla\bfU\right)-\bff_\lambda(\bfb_\alpha),\bfphi\right)\,,\ \ \forall\,\bfphi\in\cald(\Omega)\,,\\
\theta&=\hat{\lambda}F_\lambda(\bfu;\bfb_\alpha)\,,
\ea\right.
\eeq{18}  
where $[\cdot,\cdot]\equiv [\cdot,\cdot]_{\Omega}$ and $(\cdot,\cdot)\equiv (\cdot,\cdot)_{\Omega}$. 
\end{itemize}


\EDD{1}
\Br Observing that $\cald^{1,2}_0(\Omega)\subset W^{1,2}({\Omega}_R)$, for all $R>R_*$, and taking into account the properties of $\bfU$ listed in \lemmref{1} and those of $\bfh$ in \eqref{h}, it is easy to show that all integrals in \eqref{18} are well-defined.
\ER{1}
\par
We are now ready to state:
\Bp
For any given $\lambda\in\real$, there exists at least one weak solution to \eqref{17}. This solution obeys the estimate
$$
\|\nabla\bfu\|_2+|\theta|\le C(\lambda)
$$
with $C(\lambda)$ depending only on $\Omega$ and $\lambda$. 
\label{prop:1}
\EP{1}

\par
The proof of \propref{1} will be achieved  by the ``invading domains" approach, tracing back to the work of {\sc Leray} \cite[\S\,17]{Leray}. As is well-known, this method relies on the following strategy. We  first formulate  a suitably modified version of problem \eqref{17} on each element of an increasing sequence of bounded domains, $\{\Omega_n\}$, whose union coincides with $\Omega$. Then, we will prove that the sequence of corresponding solutions, in appropriate norms, is uniformly bounded with respect to $n$. The final step consists in securing that a cluster point of this sequence  is indeed  a weak solution to the original problem \eqref{17}.
With this strategy in mind, the next subsection will be entirely devoted to the study of the truncated problem, while Section \ref{sec:2.2} deals with the limit as $n\to\infty$, that completes the proof of \propref{1}.

 We conclude this preliminary section by observing that, actually, once the validity of \propref{1} is established, Theorem \ref{th:1} {directly follows as we show next.}

\begin{proof}[Proof of Theorem \ref{th:1}] {It is standard to check that } $\bfv:=\bfu-\bfb_\alpha+\bfU(\bfb_\alpha)$ is a weak solution to the problem \eqref{7}$_1$
	in the sense of \cite[Definition X.1.1]{Gab}. Therefore, the $C^\infty$-property follow from \cite[Theorem X.1.1]{Gab}. Moreover, \cite[Theorem X.6.4]{Gab} ensures the validity of \eqref{REG}, and,  by trace theorems, that the surface integrals in \eqref{8} are well-defined. 
By using standard arguments (see \cite[Theorem X.1.1]{Gab}), one then retrieves that the weak solution $(\bfu,\theta)$  satisfies  \eqref{12}--\eqref{14} in classical sense. Finally, $\bfdelta$ is deduced from \eqref{delta}. 
\end{proof}

\subsection{Resolution in Bounded Domains} 
Let $R>0$ be such that $\Omega_R$ 
 contains the support of the function $\bfh$ defined in \eqref{h}. In
such $\Omega_R$ we then consider the following  problem 
\be\left\{\ba{lc}\medskip\left.\ba{lr}\medskip
\lambda\left((\bfu{-}\bfb_{\alpha})\cdot\nabla\bfu+\bfU\cdot\nabla\bfu+\bfu\cdot\nabla\bfU\right)& =\Delta\bfu-\nabla p+\bff_\lambda(\bfb_\alpha)\\
& \Div\bfu=0\ea\right\}\ \ \mbox{in $\Omega_R$}\\ \medskip
\bfu(\bfx)=\0 \ \mbox{at $\partial\Omega$}\,,\ \ \bfu=\0 \ \mbox{at $\partial B_R$}\,;\\
 \theta=\hat{\lambda}\,F_\lambda(\bfu;\bfb_\alpha)\,.
\ea\right.
\eeq{24}
In analogy with \defref{1},  we give the following definition of weak solution to problem \eqref{24}.
\Bd The pair $(\bfu,\theta)$ is a {\em weak solution} to \eqref{24} if
\begin{itemize}
\item[{\rm (i)}] $(\bfu,\theta)\in \cald_0^{1,2}(\Omega_R)\times\real$\,;
\item[{\rm (ii)}] $(\bfu,\theta)$ satisfies the following equations ($[\cdot,\cdot]\equiv [\cdot,\cdot]_{\Omega_R}$; $(\cdot,\cdot)\equiv (\cdot,\cdot)_{\Omega_R}$):
\be\ba{rl}\medskip
\left[\bfu,\bfphi\right]&=-\left(\lambda\left((\bfu{-}\bfb_\alpha)\cdot\nabla\bfu+\bfU\cdot\nabla\bfu+\bfu\cdot\nabla\bfU\right)-\bff_\lambda(\bfb_\alpha),\bfphi\right)\,,\ \ \forall\,\bfphi\in\cald^{1,2}_0(\Omega_R)\,,\\
\theta&=\hat{\lambda}F_\lambda(\bfu;\bfb_\alpha)\,.
\ea
\eeq{25}   
\end{itemize}
\EDD{2} 
Using a fixed point argument, we shall  show existence of weak solutions to \eqref{24} for arbitrary $\lambda$, which, in addition, satisfy suitable bounds, uniformly with respect to $R$. This will allows us to let $R\to\infty$ along a sequence and prove that the corresponding solutions (along a subsequence) will converge  to a weak solution of the original problem \eqref{17}. The accomplishment of the last two properties constitutes the crucial point of the entire approach.\par 
In order to implement the above procedure, we begin to rewrite \eqref{25} as an operator equation in $\cald_0^{1,2}(\Omega_R)\times\real$. In this respect, we need the following result.
\Bl  There exists a map
$$
\bfN:(\bfu,\theta)\in \cald_0^{1,2}(\Omega_R)\times \real\mapsto\bfN(\bfu,\theta)\in\cald_0^{1,2}(\Omega_R)
$$
such that $\bfu$ solves \eqref{25}$_1$ if and only if 
\be
\bfu=\bfN(\bfu,\theta)\ \ \mbox{in $\cald_0^{1,2}(\Omega_R)$}\,.
\eeq{26}
\EL{NL}%
\begin{proof}
By using H\"older  and Sobolev inequalities, we get
$$
|(\bfu\cdot\nabla\bfu,\bfphi)|\le \|\bfu\|_4\|\nabla\bfu\|_2\|\bfphi\|_4\le c\|\nabla\bfu\|_2^2\|\nabla\bfphi\|_2\,. 
$$
Likewise, also with the help of (iv) and (i) in \lemmref{1} and Poincar\'e inequality, we show
$$
|\left((\bfb_\alpha{-}\bfU)\cdot\nabla\bfu+\bfu\cdot\nabla\bfU,\bfphi\right)|\le \big(1+c_1\sup_{\bfx\in\Omega_{\rho_0}}(|\bfU(\bfx)|+|\nabla\bfU(\bfx)|)\|\nabla\bfu\|_2\big)\|\bfphi\|_2\le c_2\,\|\nabla\bfu\|_2\|\nabla\bfphi\|_2\,,
$$
with $c_1,c_2$ depending only on $\Omega$. Finally, recalling the definition \eqref{14} and employing again (i) and (iv) in \lemmref{1}, in conjunction with Sobolev and Poincar\'e inequalities, we deduce
$$
|(\bff_\lambda(\bfb_\alpha),\bfphi)|\le 	\lambda\sup_{\bfx\in\Omega_{\rho_0}}(|\bfU(	\bfx)|+|\nabla\bfU(\bfx)|+\tfrac{1}\lambda|D^2\bfU(\bfx)|)\|\bfphi\|_2\le c_3 \|\nabla\bfphi\|_2\,,
$$
where $c_3$ depends on $\lambda,\Omega$ . In view of all the above, we may conclude that the right-hand side of \eqref{25}$_1$ defines a (linear) bounded functional on $\cald_0^{1,2}(\Omega_R)$, if $(\bfu,\theta)\in\cald_0^{1,2}(\Omega_R)\times \real$. Consequently, by Riesz Theorem, for any such $(\bfu,\theta)$, there exists an element $\bfN(\bfu,\theta)\in\cald_0^{1,2}(\Omega_R)$ such that
\be
[\bfN(\bfu,\theta),\bfphi]=-\left(\lambda\left((\bfu{-}\bfb_\alpha)\cdot\nabla\bfu+\bfU\cdot\nabla\bfu+\bfu\cdot\nabla\bfU\right)-\bff_\lambda(\bfb_\alpha),\bfphi\right)\,,\ \ \forall\,\bfphi\in\cald^{1,2}_0(\Omega_R)\,.
\eeq{27}
The latter, in combination with \eqref{25}$_1$, proves the lemma.
\end{proof}
Let  
$$
X:=\cald_0^{1,2}(\Omega_R)\times\real\,, 
$$
and set
$$
N: X\times\real\to \real: (\bfu,\theta)\mapsto N(\bfu,\theta):=\hat{\lambda}F_\lambda(\bfu;\bfb_\alpha)
$$
with $F$ defined in \eqref{16}. Now,
in view of \lemmref{NL}, we can  write \eqref{25} in an operator form.  In fact, let
$$
{\sf u}:=\left(\ba{cc}\medskip \bfu\\ \theta\ea\right)\,,
$$
and 
$$
{\sf M}: X\to X: {\sf u}\mapsto  {\sf M}({\sf u}):=\left(\ba{cc}\medskip \bfN(\bfu,\theta)\\ N(\bfu,\theta)\ea\right)\,.
$$
Then, \eqref{25} can be equivalently written as follows
\be
{\sf u}={\sf M}({\sf u})\ \ \mbox{in $X$.}
\eeq{28}
Our next objective is to show that, for any assigned $\lambda\in\real$, the equation \eqref{28} has at least one corresponding solution, namely, there exists at least one weak solution to \eqref{24}. In this regard, we need  the following property.
\Bl
The operator ${\sf M}$ is compact. \label{com}
\EL{com}
\begin{proof} 
In the whole proof we denote by $c>0$ a constant that depends at most on $\lambda$ and $\Omega$, but it is independent on $n\in\nat$. 
	
	Let $\{{\sf u}_n\}$ be a bounded sequence in $X$. Taking into account that, in view of Poincar\'e inequality, $\|\nabla(\cdot)\|_2$ and $\|\cdot\|_{1,2}$ are equivalent norms in $\cald_0^{1,2}(\Omega_R)$, this implies that there is $C>0$ independent of $n\in\nat$ such that \footnote{Throughout the proof we will not make a notational difference between a sequence and any of its subsequences.} 
\be
{\|\nabla \bfu_n\|_{2}}+|\theta_n|\le C\,.
\eeq{30}
For $n,m\in\nat$, we introduce the following notation
\be
\bfu_{n,m}:=\bfu_n-\bfu_m\,,\ \bfb_n:=\bfb_\alpha(\theta_n)\,,\ \bfb_{n,m}:=\bfb_n-\bfb_m\,,\ \bfU_n:=\bfU(\bfb_n)\,,\  \bfU_{n,m}:=\bfU_n-\bfU_m\,. 
\eeq{un}
Then,
from \eqref{27}, after integrating by parts over $\Omega_R$, and recalling \eqref{14} we deduce
\be\ba{rl}\medskip
[\bfN(\bfu_n,\theta_n)-\bfN(\bfu_m,\theta_m),\bfphi] =& \lambda((\bfu_{n,m}-\bfb_{n,m}+\bfU_{n,m})\cdot\nabla\bfphi,\bfu_n)+\lambda((\bfu_{m}{-}\bfb_m+\bfU_{m})\cdot\nabla\bfphi,\bfu_{n,m})
\\ \medskip
&
-\lambda((\bfU_{n,m}{-}\bfb_{n,m})\cdot\nabla\bfphi,\bfU_n)-\lambda((\bfU_m{-}\bfb_m)\cdot\nabla\bfphi,\bfU_{n,m})+(\nabla\bfU_{n,m},\nabla\bfphi)\\
= &\Sum{i=1}{5}I_i
\,.
\ea
\eeq{31}
If we employ on the right-hand side of \eqref{31}  H\"older inequality, the assumption \eqref{30}, the embedding $L^4\subset W^{1,2}$, and \lemmref{1}-(iv), we show
\be\ba{rl}\medskip
|I_1|+|I_2|&\le c\,(\|\bfu_{n,m}\|_4^2+\|\bfu_{n,m}\|_4+\|\bfU_{n,m}\|_{1,2}+|\bfb_{n,m}|)\|\nabla\bfphi\|_2\\
|I_3|+|I_4|+|I_5|&\le c\,(\|\bfU_{n,m}\|_{1,2}+|\bfb_{n,m}|)\|\nabla\bfphi\|_2\,.
\ea
\eeq{32}
By \eqref{30}, it follows that there exists a subsequence $\{\theta_n\}$ that is Cauchy in $\real$.
Therefore, from \eqref{9_0} and \lemmref{1}, we deduce that $\{\bfb_n\}$ and $\{\bfU_n\}$ are also Cauchy in $ S^2$ and $W^{1,2}$, respectively. Likewise, again  \eqref{30} and the compactness of the embedding $L^4(\Omega_R)\subset W^{1,2}(\Omega_R)$, entail that $\{\bfu_n\}$ is Cauchy in $L^4(\Omega_R)$. Thus, from \eqref{30}--\eqref{32}, we obtain that for any $\varepsilon>0$ we may pick $n,m$ sufficiently large so that
$$
|[\bfN(\bfu_n,\theta_n)-\bfN(\bfu_m,\theta_m),\bfphi]|\le c\,\varepsilon\,\|\nabla\bfphi\|_2\,,\ \,\forall\,\bfphi\in\cald_0^{1,2}(\Omega_R)\,.
$$ 
Choosing in the latter $\bfphi=\bfN(\bfu_n,\theta_n)-\bfN(\bfu_m,\theta_m)$ proves that $\bfN$ maps bounded sets into relatively compact sets. Since, by the same token, we can show the continuity of $\bfN$, we may conclude that $\bfN$ is compact. The proof that also $N$ is compact is performed exactly along the same lines above, and we leave the details to the reader.
\end{proof}
To prove existence to \eqref{28}, we shall employ the following result  due to {Schaefer}  \cite{Sh};
see also \cite[Theorem 6.A]{Zei}.
\Bl Let $X$ be a Banach space and  
$M:X\mapsto X$ be a compact map. Suppose that the set 
$$
\{x\in X: x=t\,M(x)\ \mbox{for some}\ t\in(0,1)\}
$$
is bounded. Then $M$ has a fixed point. \label{Sh}
\EL{Sh}

We are now in a position to show the main result of this section.
\Bp For any $\lambda\in\real$, there exists at least one weak solution to \eqref{24}. This solution satisfies the estimate
\be
\|\nabla\bfu\|_{2,\Omega_R}+|\theta|\le C\,,
\eeq{33}
where $C$ depends only on $\Omega$ and $\lambda$. 
\EP{R}
\begin{proof} In view of Lemmas  \ref{com} and \ref{Sh}, we only have to show that all possible solutions to the equation ${\sf u}=t\,{\sf M}({\sf u})$, $t\in(0,1)$ are bounded with a bound independent of $t$. Now, such an equation is equivalent to the following ones
\be\ba{rl}\medskip
\left[\bfu,\bfphi\right]&=-t\,\left(\lambda\left(\left(\bfu{-}\bfb_\alpha\right)\cdot\nabla\bfu+\bfU\cdot\nabla\bfu+\bfu\cdot\nabla\bfU\right)-\bff_\lambda(\bfb_\alpha),\bfphi\right)\,,\ \ \forall\,\bfphi\in\cald^{1,2}_0(\Omega_R)\,,\\
\theta&=t\,\hat{\lambda}F_\lambda(\bfu;\bfb_\alpha)\,.
\ea
\eeq{34}
Choosing $\bfphi=\bfu$ in \eqref{34}$_1$ and integrating by parts over $\Omega_R$, we deduce
\be
\|\nabla\bfu\|_2^2=t\,\left[\lambda(\bfu\cdot\nabla\bfu,\bfU)+(\bff_\lambda(\bfb_\alpha),\bfu)\right]\,.
\eeq{35}
From \lemmref{1}-(v),
$$
|(\bfu\cdot\nabla\bfu,\bfU)|\le\half\|\nabla\bfu\|_2^2\,.
$$
Further, recalling \eqref{14},  using  Sobolev embedding inequalities in conjunction with \lemmref{1}-(iv), and Poincar\'e inequality, we get
\be
\left|(\bff_\lambda(\bfb_\alpha),\bfu)\right|\le c\|\bfU\|_{2,2}(\lambda+\|\bfU\|_{2,2})\|\bfu\|_{2,\Omega_{\rho_0}}\le c_{1\lambda}\|\nabla\bfu\|_2\,,
\eeq{36}
where $c$ and $c_{i\lambda}$, $i=1,2\ldots$, here and in the rest of the proof  denote positive constants depending respectively only on $\Omega$ and only on $\Omega,\lambda$. Employing the above estimates in \eqref{35} and since $t\in(0,1)$, we conclude
\be
\|\nabla\bfu\|_2\le 2c_{1\lambda}\,.
\eeq{37}
Likewise, integrating by parts as necessary in \eqref{16} and using H\"older inequality, we show
$$
|F_\lambda(\bfu;\bfb_\alpha)|\le \lambda(\|\bfu\|_{4,\hat{\Omega}}^2+1+\|\bfU\|_4^2+\|\nabla\bfu\|_2)\|\nabla\bfh\|_2 +|(\bff_\lambda(\bfb_\alpha),\bfh)|\,.
$$
Thus, arguing as in \eqref{36}, and using again \lemmref{1}-(iv),   Sobolev embedding, and \eqref{37} we infer
$$
|F_\lambda(\bfu;\bfb_\alpha)|\le c_{2\lambda}\|\nabla\bfh\|_2:= c_{3\lambda}\,. 
$$
From this inequality and \eqref{34}$_2$ we then conclude
\be
|\theta|\le \hat{\lambda}c_{3\lambda}\,,
\eeq{38}
and the proposition follows from \eqref{37} and \eqref{38}.
\end{proof}
\subsection{Proof of Proposition \ref{prop:1}}\label{sec:2.2}
Let $\{\Omega_n\equiv\Omega_{R_n}\}$, $R_1$ sufficiently large, be a sequence of ``invading domains," namely,
$$
\Omega_{n-1}\subset\Omega_{n}\,,\ n\in\nat\,;\ \,\ \cup_{n=1}^\infty\Omega_n=\Omega\,,
$$
and let $\{(\bfu_n,\theta_n)\}$ be the sequence of corresponding  weak solutions determined in \propref{1}. For each $n$, we extend $\bfu_n$ to 0 outside $\Omega_n$ and continue to denote by $\bfu_n$ its extension. Consequently,  $\{\bfu_n\}\subset W_0^{1,2}(\Omega)$. Using the bound \eqref{35} and the compact embedding $L^4(\Omega_R)\subset W_0^{1,2}(\Omega)$, we deduce that there exist a subsequence, still denoted by  $\{(\bfu_n,\theta_n)\}$, and $(\bfu,\theta)\in\cald_0^{1,2}(\Omega)\times\real$ such that
\be\ba{ll}\medskip
\bfu_n\rightharpoonup \bfu\,\ \mbox{weakly in $\cald_0^{1,2}(\Omega)$}\,;\ \theta_n\to\theta\,\ \mbox{in $\real$}\,;\\
\bfu_n\to\bfu\,\ \mbox{strongly in $L^4(\Omega_R)$ for all $R>R_*$.}
\ea
\eeq{41}
From \eqref{25} it follows that, for any fixed $\bfphi\in\cald(\Omega)$ and $n$ large enough, the sequence 
$\{(\bfu_n,\theta_n)\}$ satisfies
 \be\ba{rl}\medskip
\left[\bfu_n,\bfphi\right]&=-\left(\lambda\left((\bfu_n-V\bfb_n)\cdot\nabla\bfu_n+\bfU_n\cdot\nabla\bfu_n+\bfu_n\cdot\nabla\bfU_n\right)-\bff_\lambda(\bfb_n),\bfphi\right)\,,\\
\theta_n&=\hat{\lambda}F_\lambda(\bfu_n;\bfb_n)\,,
\ea
\eeq{43}  
where we adopted the notation in \eqref{un}. We now let $n\to\infty$ in \eqref{43}. Taking into account that $\supp(\bfphi)$ is compact, we can repeat word by word the same argument used in the proof of \lemmref{com}, by formally replacing there $(\bfu_m,\theta_m)$ with $(\bfu,\theta)$, and show that the limit quantity $(\bfu,\theta)$ satisfies \eqref{18}. The proof of the proposition is therefore completed. 
\qed

\section{Uniqueness of Equilibrium}\label{sec:uniqueness}
Once the existence of equilibria for system \eqref{01_nrs} (i.e., the solutions to \eqref{7}--\eqref{8}) is established,  one naturally wonders whether the equilibrium position is unique. This also because the method we use, for the construction of such solutions, involves an auxiliary function $\bfh$ whose choice can be made in infinitely many different ways; see \eqref{h1}. In this section, we address this question.
\par The main result is Theorem \ref{th:unicita}, which states that a smallness assumption on the parameter $\lambda$ is sufficient  to ensure uniqueness of solutions to \eqref{7}--\eqref{8}. The proof of Theorem \ref{th:unicita} relies on a key observation: the structure of the fluid equations in the system satisfied by the difference between two solutions is of Oseen-type, with an Oseen velocity $\bfb_{\alpha}$. Thus, to handle this structure, we adapt techniques developed for the Oseen equations in exterior domains, as presented in \cite[Section VIII]{Gab}, to the fluid-body setting considered here.   A preliminary step in applying these techniques is to obtain $\lambda$-dependent bounds for certain $L^q$-norms of the solutions to \eqref{7}--\eqref{8}. This is achieved in Lemma \ref{Lq-estimate}, which in turn allows us to establish Theorem \ref{th:unicita}, which we will prove in Section \ref{sub2}. 

\par At this point, we state:
\begin{theorem}\label{th:unicita}
	There exists $\lambda_0>0$ depending on $\Omega$, such that, if 
	$
	\lambda<\lambda_0\,,
	$
	the solution to \eqref{7}--\eqref{8} given by Theorem \ref{th:1} is unique.
\end{theorem}
%

Before proceeding to the proof of Theorem \ref{th:unicita}, which will be unravelled in Sections \ref{sub1}-\ref{sub2}, we need some technical construction. In  what follows, we will indeed use  a lifting for the boundary data in \eqref{10} which differs from the one that we constructed at the beginning of Section \ref{sec:equilibria}.  Let $\varphi$ be the cut-off function used to define $\bfh$ in \eqref{h1}. For any $\theta\in \mathbb{R}$, take $\bsfU(\bfx;\bfb_{\alpha}(\theta))$ with $\bsfU$ as in \eqref{U}. For some fixed $\rho_0>R_*$, we define 
\be
\bfV(\bfx,\bfb_{\alpha}(\theta)):=\text{curl}\left(\varphi({|\bfx|}/{\rho_0})\bsfU\right)\,.
\eeq{Lifting}
Clearly, $\bfV$ satisfies 
\be
\text{div}\bfV=0\,\,\text{in}\,\,\Omega\,,\quad \bfV\equiv0\,\,\text{for} \,\,|\bfx|>2\rho_0\,, \quad \bfV=\bfb_\alpha(\theta)\,\ \mbox{at $\partial\Omega$}\,,\quad  \|\bfV\|_{C^k(\bar{\Omega})}\le C \,,\ \, \forall k\in\mathbb{N}_0\,,
\eeq{V_stime}
for some $C>0$ depending on $R_*$.
\subsection{Bounds for the Solutions in terms of $\lambda$}\label{sub1}
In order to prove uniqueness, as already mentioned, we begin to show that some suitable norms of the solutions to \eqref{7}-\eqref{8} given by Theorem \ref{th:1} can be bounded in terms of $\lambda$. 
\begin{lemma}\label{Lq-estimate}
	Let $1<s<2$, set 
\be
a_1=\min\{1,\lambda^{1/2}\}\,,\ \, a_2=\min\{1,\lambda^{1/4}\}\,,
\eeq{aa}  and let $(\bfv,p,\bfdelta,\theta)$ be any solution  to \eqref{7}-\eqref{8} given by Theorem \ref{th:1}. 
	There exists $\lambda_0=\lambda_0(\Omega,s)$ such that, if
\be
\lambda<\lambda_0\,,
	\eeq{smallness1}
then $\bfv$ satisfies 
	\be
	\begin{aligned}
\|\nabla \bfv \|_2\le C(1+\lambda)\,,\qquad {a_1}\|\bfv+\bfb_{\alpha}(\theta)\|_4\le C\,(1+\lambda)\,,  \qquad  {a_2}\|\nabla\bfv \|_{\frac{4s}{4-s}}\le C\,(1+\lambda) \,,
\end{aligned}
\eeq{bound_norms}
where $C=C(\Omega)$.
\end{lemma}  
\begin{proof}
{Let $\rho_0>R_*$ be given, $C_0>0$ be the Poincaré constant associated to $\Omega_{\rho_0}$ and
assume $\lambda C_0^2<1-\tau$ for some $\tau>0$. }  Let $(\bfv,p,\bfdelta,\theta)$ be a solution to \eqref{7}-\eqref{8} given by Theorem \ref{th:1}. We start proving the first inequality in \eqref{bound_norms}. Define $\bfu:=\bfv+\bfb_{\alpha}(\theta)-\bfV$, with  $\bfV=\bfV(\bfb_{\alpha}(\theta))$ as in \eqref{Lifting}, then $\bfu$ satisfies 
\be\left\{\ba{ll}\medskip\left.\ba{lr}\medskip
\lambda\left[(\bfu+\bfV-\bfb_\alpha(\theta))\cdot\nabla\bfu+\bfu\cdot\nabla\bfV\right] -\Delta \bfV+\lambda\left[  \bfV\cdot\nabla\bfV-\bfb_{\alpha}(\theta)\cdot\nabla\bfV\right] = &\Delta\bfu-\nabla p \\
&\Div\bfu=0\ea\right\}\ \ \mbox{in $\Omega$}\\
\bfu=\0 \ \mbox{at $\partial\Omega$}\,;\ \ \Lim{|\bfx|\to\infty}\bfu=\0\,.
\ea\right.
\eeq{pb_u}
Dot-multiplying \eqref{pb_u}$_1$ by $\bfu$ and integrating by parts over $\Omega$ yields 
\be
\|\nabla\bfu\|^2_2=\int_{\Omega}(-\Delta\bfV+\lambda \bfV\cdot\nabla\bfV-\lambda\bfb_{\alpha}(\theta)\cdot\nabla\bfV+\lambda\bfu\cdot\nabla\bfV)\cdot\bfu\,.
\eeq{above}
Applying H\"older inequality in \eqref{above} implies
$$
 \|\nabla\bfu\|^2_2\le C_0\left(\sup_{\bfx\in\Omega_{\rho_0}}(|\Delta\bfV(\bfx)|+\lambda|\bfV(\bfx)||\nabla\bfV(\bfx)|+\lambda|\bfb_{\alpha}||\nabla\bfV(\bfx)|)\right)\|\nabla\bfu\|_2+\lambda\, C^2_0\|\nabla \bfu\|^2_2\,,
$$
so that 
{
$$
\tau\|\nabla\bfu\|_2\le C_0\left(\sup_{\bfx\in\Omega_{\rho_0}}(|\Delta\bfV(\bfx)|+\lambda|\bfV(\bfx)||\nabla\bfV(\bfx)|+\lambda|\bfb_{\alpha}||\nabla\bfV(\bfx)|)\right).
$$
Combined with \eqref{V_stime}, this yields the estimate
\be
\|\nabla \bfv\|_2\le \|\nabla\bfu\|_2+\|\nabla\bfV\|_2\le c_1(1+\lambda)
\eeq{u_V}
for some $c_1>0$ depending only on $\Omega$.}
\par  The proof of the other bounds appearing in \eqref{bound_norms} follows from the ideas of \cite[Lemma X.7.1]{Gab}. We begin to prove that there exists a solution, say $(\bfv',p')$, to \eqref{7}, in a suitable regularity class, satisfying \eqref{bound_norms}. If we define $\bfw:=\bfv'+\bfb_{\alpha}(\theta)$, then $(\bfw,p)$ must satisfy 
	\be\left\{\ba{cc}\medskip\left.\ba{ll}\medskip
	{\lambda}\,(\bfw-\bfb_{\alpha}(\theta))\cdot\nabla\bfw=\Delta\bfw-\nabla p\\
	\qquad\qquad\qquad\qquad\qquad\Div\bfw=0\ea\right\}\ \ \mbox{in $\Omega$}\\
	\bfw(\bfx)=\bfb_\alpha(\theta) \ \mbox{at $\partial\Omega$}\,;\ \ \Lim{|\bfx|\to\infty}\bfw(\bfx)=\bf0\,.
	\ea\right.
	\eeq{12_U}

 Notice that $\theta$ is fixed, as it is the angle associated to the solution $(\bfv,p,\bfdelta,\theta)$, hence from now on we will just write $\bfb_{\alpha}$, dropping the dependence on $\theta$.
We consider a sequence of approximating solutions $\{\bfw_k,p_k\}$ to \eqref{12_U}, defined by $\bfw_0\equiv p_0\equiv 0$ and
		\be\left\{\ba{lc}\medskip\left.\ba{lr}\medskip
\Delta\bfw_k+\lambda \bfb_\alpha\cdot\nabla\bfw_{k}=&\nabla p_k+\lambda\left(\bfw_{k-1}\cdot\nabla\bfw_{k-1}\right)\\
&	\Div\bfw_k=0\ea\right\}\ \ \mbox{in $\Omega$}\\
	\bfw_k(\bfx)=\bfb_{\alpha} \ \mbox{at $\partial\Omega$}\,;\ \ \Lim{|\bfx|\to\infty}\bfw_k(\bfx)=\0\,,
	\ea\right.
	\eeq{12_k}
for $k\ge 1$. {For $k=1$, \eqref{12_k} corresponds to an Oseen equation, so that \cite[Theorem VII.7.1]{Gab} combined with \cite[(II.6.22)]{Gab} ensure existence and uniqueness of $\{\bfw_1,p_1\}$ that satisfies}
	$$
\bfw_1\in L^{\frac{2s}{2-s}}(\Omega)\cap D^{1,\frac{4s}{4-s}}(\Omega)\cap D^{2,s}(\Omega)\,, \qquad p_1\in D^{1,s}(\Omega)\,,	\qquad \text{with }1<s<2
	$$
satisfying 
	\be
	\begin{aligned}
a_1\|\bfw_1\|_{\frac{2s}{2-s}}+\frac{1}\gamma_1\|\nabla\bfw_1\|_{\frac{3s}{3-s}}+a_2\|\nabla\bfw_1\|_{\frac{4s}{4-s}}+\|D^2\bfw_1\|_{s}+\frac{1}\gamma_1\|p_1\|_{\frac{3s}{3-s}}+\|\nabla p_1\|_{s}&\le c_2\,,
\end{aligned}
\eeq{f}
{where $\gamma_1>0$ (and is independent of $\bfw_1$ and $p_1$), $c_2>0$ depends on $s$ and $\Omega$ only, 
$$
	a_1=\text{min}\{1,\lambda^{1/2}\}\,,\qquad a_2=\text{min}\{1,\lambda^{1/4}\}\,.
	$$}
{We have implicitly used the fact that \(\bfb_\alpha(\theta)\in S^2\) so that \(\|\bfb_\alpha(\theta)\|_{2-1/s,s,\partial \Omega}\) is a real number that depends only on $s$ and $\Omega$.}
 We next observe that if we  write $\bfw_k=\bfu_k+\bfV$, by the same type of arguments that lead to \eqref{u_V}, we have
\be
\|\nabla\bfw_k\|_2\le c_1(1+\lambda) \qquad \forall\, k\ge 1\,,
\eeq{wk_V}
thus we can add such bound to \eqref{f} and we obtain 
	\be
\begin{aligned}
&	a_1\|\bfw_1\|_{\frac{2s}{2-s}}+\|\nabla \bfw_1\|_2+\frac{1}\gamma_1\|\nabla\bfw_1\|_{\frac{3s}{3-s}}+a_2\|\nabla\bfw_1\|_{\frac{4s}{4-s}}+\|D^2\bfw_1\|_{q}+\frac{1}\gamma_1\|p_1\|_{\frac{3s}{3-s}}+\|\nabla p_1\|_{s}\\&\le c_2+c_1(1+\lambda)\le
{ c_3
(1+\lambda)
\,.}
\end{aligned}
\eeq{allD}
We will next show by induction that  the estimate
	\be
\begin{aligned}
	&a_1\|\bfw_k\|_{\frac{2s}{2-s}}+\|\nabla \bfw_k\|_2+\frac{1}\gamma_1\|\nabla\bfw_k\|_{\frac{3s}{3-s}}+a_2\|\nabla\bfw_k\|_{\frac{4s}{4-s}}+\|D^2\bfw_k\|_{s}+\frac{1}\gamma_1\|p_k\|_{\frac{3s}{3-s}}+\|\nabla p_k\|_{s}\\&\le 2c_3{(1+\lambda)}
\end{aligned}
\eeq{hp_ind}
holds for all $k\in\mathbb{N}$, provided that $\lambda$ is small enough. Suppose therefore that \eqref{hp_ind} is true
for some $k\in\mathbb{N}$. Then, H\"older inequality implies that 
$$
\|\bfw_k\cdot\nabla\bfw_{k}\|_{s}\le \|\bfw_{k}\|_{\frac{2s}{2-s}}\|\nabla \bfw_k\|_2\,.
$$
and from the induction hypothesis \eqref{hp_ind}, we deduce
$$
\|\bfw_k\cdot\nabla\bfw_{k}\|_s\le \frac{4c_3^2(1+\lambda)^2}{a_1}\,.
$$
Hence, we can apply again \cite[Theorem VII.7.1]{Gab}-\cite[(II.6.22)]{Gab} to \eqref{12_k}  and \db{use } \eqref{wk_V} to  obtain that $\{\bfw_{k+1},p_{k+1}\}$  satisfies 
{	\be
\begin{aligned}
	&a_1\|\bfw_{k+1}\|_{\frac{2s}{2-s}}+\|\nabla\bfw_{k+1}\|_2+\frac{1}\gamma_1\|\nabla\bfw_{k+1}\|_{\frac{3s}{3-s}}+a_2\|\nabla\bfw_{k+1}\|_{\frac{4\db{?q?}}{4-s}}+\|D^2\bfw_{k+1}\|_{s}+\frac{1}\gamma_1\|p_{k+1}\|_{\frac{3s}{3-s}}+\|\nabla p_{k+1}\|_{s}\\&\le c_2\left(\frac{4\lambda c_3^2(1+\lambda)^2}{a_1}+1\right)+c_1(1+\lambda)\le \left(\frac{\lambda}{a_1}4c_2c_3(1+\lambda)+1\right)c_3(1+\lambda)\,.
\end{aligned}
\eeq{tesi_ind}}
It follows that \eqref{hp_ind} is satisfied for all $k\in\mathbb{N}$ as soon as {
\be
\frac{\lambda}{\min\{1,\lambda^{1/2}\}}(1+\lambda)<\frac{1}{4c_2c_3}\,.
\eeq{th}
}
We will next prove that $\{\bfw_k,p_k\}$ is a Cauchy sequence in 
$$
\mathcal{S}\equiv \left(L^\frac{2s}{2-s}(\Omega)\cap \dot{D}^{1,\frac{4s}{4-s}}(\Omega)\cap \dot{D}^{2,s}(\Omega)
\right)\times \dot{D}^{1,s}(\Omega)\,,\qquad 1<s<2\,,
$$
where by $\dot{D}^{m,q}(\Omega)$, $1\le q\le \infty, m\in\mathbb{N}$, we indicate the space of all (equivalent classes) $\bfw\in D^{m,q}$ such that $\bfw=\bfu+\mathcal{P}$, for some $\bfu\in D^{m,q}$ and some $\mathcal{P}\in {\sf P}_m$, {where ${\sf P}_m$ is the class of all polynomials of degree less or equal than $m-1$}. 
This is indeed a Banach space  if equipped with the norm $|\cdot|_{m,s}$ (see \cite[Lemma II.6.2]{Gab}). 

Take $\{\bfw_{k},p_{k}\}$ and $\{\bfw_{k+1},p_{k+1}\}$. Then, the problem satisfied by the difference between the two solutions reads as
	\be\left\{\ba{l}
	\medskip\left.\ba{rl}\medskip
 \Delta (\bfw_{k+1}-\bfw_k)-\lambda\bfb_{\alpha}\cdot \nabla (\bfw_{k+1}-\bfw_k)=&\nabla(p_{k+1}-p_k) +\lambda(\bfw_{k}-\bfw_{k-1})\cdot\nabla\bfw_{k})\\
 & +\lambda\bfw_{k-1}\cdot\nabla(\bfw_{k}-\bfw_{k-1})\\[6pt]
&\hspace{2.8cm} \Div (\bfw_{k+1}-\bfw_{k})=0
\ea\right\}\ \ \mbox{in $\Omega$}\\
(\bfw_{k+1}-\bfw_{k})(\bfx)=\0\ \mbox{at $\partial\Omega$}\,;\ \ \Lim{|\bfx|\to\infty}(\bfw_{k+1}-\bfw_{k})(\bfx)=\0\,.
\ea\right.
\eeq{diff_k}
Set $\bff = (\bfw_{k}-\bfw_{k-1})\cdot\nabla\bfw_{k}+\bfw_{k-1}\cdot\nabla(\bfw_{k}-\bfw_{k-1})$.
The application of \cite[Theorem VII.7.1]{Gab}-\cite[(II.6.22)]{Gab} to \eqref{diff_k} gives then 
$$
\begin{aligned}
&a_1\|\bfw_{k+1}-\bfw_{k}\|_{\frac{2s}{2-s}}+\frac{1}\gamma_1\|\nabla(\bfw_{k+1}-\bfw_{k})\|_{\frac{3s}{3-s}}+a_2\|\nabla(\bfw_{k+1}-\bfw_{k})\|_{\frac{4s}{4-s}}+\|D^2(\bfw_{k+1}-\bfw_{k})\|_s\\&+\frac{1}\gamma_1\|p_{k+1}-p_k\|_{\frac{3s}{3-s}}+\|\nabla(p_{k+1}-p_k)\|_s\le  c_2\lambda\|\bff\|_s\,.
\end{aligned}
$$
By Minkowski and H\"older inequality, we infer
\be
\begin{aligned}
&\|\bfw_{k-1}\cdot\nabla (\bfw_k-\bfw_{k-1})+(\bfw_k-\bfw_{k-1})\cdot \nabla \bfw_k\|_s\\&\le \|\bfw_{k-1}\|_4\|\nabla (\bfw_k-\bfw_{k-1})\|_{\frac{4s}{4-s}}+\|\nabla\bfw_k\|_2\|\bfw_k-\bfw_{k-1}\|_{\frac{2s}{2-s}}\,.
\end{aligned}
\eeq{RHS1-}
 Enforcing \eqref{hp_ind} with $s=4/3$ in \eqref{RHS1-} gives 
$$
\|\bfw_{k-1}\cdot\nabla (\bfw_k-\bfw_{k-1})+(\bfw_k-\bfw_{k-1})\cdot \nabla \bfw_k\|_s\le \frac{2{c_3}(1+\lambda)}{a_1}\left(\|\nabla (\bfw_k-\bfw_{k-1})\|_{\frac{4s}{4-s}}+a_1\|\bfw_k-\bfw_{k-1}\|_{\frac{2s}{2-s}}\right)\,.
$$ 
This inequality implies that, for all $k\ge 1$, we have 
$$
\begin{aligned}
	&a_1\|\bfw_{k+1}-\bfw_{k}\|_{\frac{2s}{2-s}}+\frac{1}{\gamma_1}\|\nabla(\bfw_{k+1}-\bfw_{k})\|_{\frac{3s}{3-s}}+a_2\|\nabla(\bfw_{k+1}-\bfw_{k})\|_{\frac{4s}{4-s}}+\|D^2(\bfw_{k+1}-\bfw_{k})\|_s\\&+\frac{1}\gamma_1\|p_{k+1}-p_k\|_{\frac{3s}{3-s}}+\|\nabla(p_{k+1}-p_k)\|_s\le {\left(\frac{2 c_2c_3\lambda(1+\lambda)}{a_1}\right)^{k+1}\,,} 
\end{aligned}
$$
which, in view of \eqref{th}, yields that $\{\bfw_k,p_k\}$ is a Cauchy sequence in $\mathcal{S}$. The estimate \eqref{hp_ind} then also holds for the limiting fields $(\bfw,p')\in \mathcal{S}$, namely
\be
a_1\|\bfw\|_{\frac{2s}{2-s}}+a_2\|\nabla \bfw\|_{\frac{4s}{4-s}}\le {2c_3(1+\lambda)\,.}
\eeq{st1}
Fix $\bfv'=\bfw-\bfb_{\alpha}$ and, if needed, restrict further the smallness assumption on $\lambda$ by imposing that
	\be
\frac{\lambda\max\{1,\lambda\}}{\min\{1,\lambda^{1/2}\}}\le \frac{1}{4c_3\gamma_1}\,.
\eeq{smallness3}
Then, using \eqref{smallness3} 
and \eqref{st1} with $s=6/5$, we infer that 
\be
\|\bfv'+\bfb_{\alpha}\|_{3}\le (2{\gamma_1\lambda})^{-1}\,.
\eeq{nice}
We finally prove that $\bfv$ coincides with $\bfv'$, which will give us the other bounds in \eqref{bound_norms}. Define $\bfu:=\bfv-\bfv',q:=p-p'$. Then $(\bfu,q)$ satisfies 
	\be\left\{\ba{lc}\medskip\left.\ba{lr}\medskip
{\lambda}\,\left(\bfu\cdot\nabla\bfu+\bfu\cdot\nabla(\bfv'+\bfb_{\alpha})+\bfv'\cdot\nabla\bfu\right)& =\Delta\bfu-\nabla q\\
& \Div\bfu=0\ea\right\}\ \ \mbox{in $\Omega$}\\
\bfu(\bfx)=\0 \ \mbox{at $\partial\Omega$}\,;\ \ \Lim{|\bfx|\to\infty}\bfu(\bfx)=\0\,.
\ea\right.
\eeq{uniq}
Multiplying \eqref{uniq} and integrating by parts over $\Omega$ yields 
\be
\|\nabla\bfu\|^2_2=-\lambda\int_{\Omega}\bfu\cdot\nabla\bfu \cdot(\bfv'+\bfb_{\alpha})\,.
\eeq{id1}
Then, thanks to the summability properties of both $\bfv$ and $\bfv'$, we infer, using H\"older inequality and \cite[II.6.22]{Gab}, that 
\be
\int_{\Omega}\bfu\cdot\nabla\bfu \cdot(\bfv'+\bfb_{\alpha})\le \|\bfu\|_6\|\nabla\bfu\|_2\|\bfv'+\bfb_{\alpha}\|_3\le \gamma_1\|\nabla\bfu\|_2^2 \|\bfv'+\bfb_{\alpha}\|_3\,.
\eeq{tri1}
Substituting \eqref{tri1} in \eqref{id1}, in view of \eqref{nice}, we infer that $\bfu=\bf0$ a.e. in $\Omega$. Therefore, $\bfv$ satisfies
$$
a_1\|\bfv+\bfb_{\alpha}\|_{\frac{2s}{2-s}}+a_2\|\nabla \bfv\|_{\frac{4s}{4-s}}\le 2c_3(1+\lambda)\,.
$$
This gives \eqref{bound_norms} and concludes the proof of the lemma. 
\end{proof}
\subsection{Proof of Theorem \ref{th:unicita}}\label{sub2}
Assume \eqref{smallness1}. 
	Consider two solutions $(\bfv_0,p_0,\bfdelta_0,\theta_0)$ and $(\bfv+\bfv_0,p+p_0,\bfdelta+\bfdelta_0,\theta+\theta_0)$ to \eqref{7}-\eqref{8} in the class \eqref{REG} given by Theorem \ref{th:1} corresponding to the same $\lambda$. Then, the quadruple $(\bfv,p,\bfdelta,\theta)$ satisfies 
	\be\left\{\ba{lc}\medskip\left.\ba{ll}\medskip
	{\lambda}\,(\bfv+\bfv_0)\cdot\nabla\bfv+\lambda \bfv\cdot \nabla\bfv_0=\Delta\bfv-\nabla p\\
\hspace{5cm}	\Div\bfv=0\ea\right\}\ \ \mbox{in $\Omega$}\\
	\bfv(\bfx)=\0 \ \mbox{at $\partial\Omega$}\,;\ \ \Lim{|\bfx|\to\infty}\bfv(\bfx)=-\bfb_\alpha(\theta+\theta_0)+\bfb_{\alpha}(\theta_0)\,,
	\ea\right.
	\eeq{7_diff}
	\be\left\{\ba{ll}\medskip
	\mathbb B(\theta+\theta_0)\cdot\bfdelta+(\mathbb B(\theta+\theta_0)-\mathbb B(\theta_0))\cdot\bfdelta_0+\mu\Int{\partial\Omega}{}\mathbb T(\bfv,p)\cdot\bfn=\0\qquad \\
k	\theta+\lambda\,\bfe_1\cdot\Int{\partial\Omega}{}\bfx\times\mathbb T(\bfv,p)\cdot\bfn=0\,.
	\ea\right.
	\eeq{8_diff}
If we define 
	$$
	\bfw:=\bfv+\bfb_\alpha(\theta+\theta_0)-\bfb_{\alpha}(\theta_0)\,,\qquad\bfw_0:=\bfv_0+\bfb_{\alpha}(\theta_0)\,,
	$$
	then $(\bfw,p,\bfdelta,\theta)$ solves 
		\be\left\{\ba{lc}\medskip\left.\ba{ll}\medskip
\lambda \bfb_{\alpha}(\theta+\theta_0)\cdot\nabla\bfw-{\lambda}\,(\bfw+\bfw_0)\cdot\nabla\bfw-\lambda \bfw\cdot \nabla\bfw_0+\lambda (\bfb_\alpha(\theta+\theta_0)-\bfb_\alpha(\theta))\cdot\nabla\bfw_0=-\Delta\bfw+\nabla p\\ 
\hspace{14cm}	\Div\bfw=0\ea\right\}\ \ \mbox{in $\Omega$}\\
	\bfw(\bfx)=\bfb_\alpha(\theta+\theta_0)-\bfb_{\alpha}(\theta_0) \ \ \ \mbox{at $\partial\Omega$}\,;\ \ \Lim{|\bfx|\to\infty}\bfw(\bfx)=\0\,,
	\ea\right.
	\eeq{7_diff_w}
	\be\left\{\ba{ll}\medskip
	\mathbb B(\theta+\theta_0)\cdot\bfdelta+(\mathbb B(\theta+\theta_0)-\mathbb B(\theta_0))\cdot\bfdelta_0+\mu\Int{\partial\Omega}{}\mathbb T(\bfw,p)\cdot\bfn=\0,\qquad \\
k\theta+\lambda\,\bfe_1\cdot\Int{\partial\Omega}{}\bfx\times\mathbb T(\bfw,p)\cdot\bfn=0\,.
	\ea\right.
	\eeq{8_diff_w}
	Fix $$\bff = (\bfw+\bfw_0)\cdot\nabla\bfw+ \bfw\cdot \nabla\bfw_0- (\bfb_\alpha(\theta+\theta_0)+\bfb_\alpha(\theta))\cdot\nabla\bfw_0$$
	and $$\bfv^*=\bfb_\alpha(\theta+\theta_0)-\bfb_{\alpha}(\theta_0).$$
Arguing as in the proof of existence, it is enough to show that $(\bfw, p)=({\bf0},0)$ and $\theta=0$ because the conclusion $\bfdelta=\bf0$ then follows from \eqref{8_diff_w}$_1$ as
	$$
	\bfdelta:=-\mu\, \mathbb{B}^{-1}(\theta_0)\cdot \Int{\partial\Omega}{}\mathbb T(\bfw,p)\cdot\bfn\,,
	$$
	and the surface integral is well-defined by Theorem \ref{th:1}.
	
	Fix now $q=\frac{4s}{4-s}$, with $1<s<\frac{4}{3}$. 	We observe that, since $(\theta+\theta_0)$ is known, we can treat \eqref{7_diff} as an Oseen-type system. Since $1<q<2$, we can then apply \cite[Theorem VII.7.1]{Gab} to \eqref{7_diff} and infer that $(\bfw,p)$ satisfies the estimate 
	\be
	a_1\|\bfw\|_{\frac{2q}{2-q}}+a_2\|\nabla\bfw\|_{\frac{4q}{4-q}}+\|D^2\bfw\|_{q}+\|\nabla p\|_{q}\le c_1(\lambda\|\bff\|_q+\|\bfv_*\|_{2-1/q,q,\partial\Omega})\,,
	\eeq{Oseen}
	where $a_1,a_2$ are  as in Lemma \ref{Lq-estimate} and $c_1$ depends on $\Omega$ only. 
	We now compute a bound for the right-hand side of \eqref{Oseen} in terms of the norms appearing on the left-hand side. We start observing that, Minkowski and H\"older inequality imply that
	\be
	\|(\bfw+\bfw_0)\cdot\nabla \bfw+\bfw\cdot \nabla \bfw_0\|_q\le \|\bfw+\bfw_0\|_4\|\nabla \bfw\|_{\frac{4q}{4-q}}+\|\nabla\bfw_0\|_2\|\bfw\|_{\frac{2q}{2-q}}\,.
	\eeq{RHS1}
	Let $r\in [q,\infty)$. In view of \cite[Exercise II.4.1]{Gab}, for any $\varepsilon>0$, there exists $c_2>0$ depending on $\varepsilon,\Omega$ such that 
	$$
	\begin{aligned}
		&	\|\nabla \bfw\|_{r,\partial \Omega}\le c_2\|\nabla\bfw\|_{q,\Omega'}+\varepsilon\|D^2\bfw\|_{q,\Omega'}\qquad\|p\|_{r,\partial \Omega}\le c_2\|p\|_{q,\Omega'}+\varepsilon\|\nabla p\|_{q,\Omega'}
	\end{aligned}
	$$
	for any bounded set $\Omega'$ such that $\Omega_0\subset\Omega'\subset\Omega$. Since $\Omega'\subset \Omega$ is bounded, inclusions between $L^p(\Omega')$-spaces hold, thus
	\be
	\|\nabla \bfw\|_{r,\partial \Omega}\le c_2\|\nabla \bfw\|_{\frac{4q}{4-q}}+\varepsilon\|D^2\bfw\|_{q}\qquad\text{ and }\qquad \|p\|_{r,\partial \Omega}\le c_2\|p\|_{\frac{3q}{3-q}}+\varepsilon\|\nabla p\|_q\,.
	\eeq{traces}
	Next, equation \eqref{8_diff}$_2$ and H\"older inequality imply that
	\be
	\begin{aligned}
 |\bfv^*|&\le c_3\lambda(\|\nabla\bfw\|_{r,\partial \Omega}+\|p\|_{r,\partial \Omega})\,,
\end{aligned}
	\eeq{trilinear_theta}
	where $c_3$ depends on $\Omega$ only.
Plugging \eqref{traces} in \eqref{trilinear_theta}, we get
		\be
	\begin{aligned}
		\|\bfv^*\cdot\nabla\bfw_0\|_q\le &\ c_4\lambda\left(\|\nabla \bfw\|_{\frac{4q}{4-q}}+\|D^2\bfw\|_{q}\right) \|\nabla \bfw_0\|_q+c_4\lambda \left(\|p\|_{\frac{3q}{3-q}}+\|\nabla p\|_q\right)\|\nabla \bfw_0\|_q\,,
	\end{aligned}
	\eeq{RHS2}
	for some $c_4$ depending on $\Omega$ only.
	The combination of \eqref{RHS1}-\eqref{RHS2} yields that 
	\be
	\begin{aligned}
	\|\bff\|_q& \le \|\nabla\bfw_0\|_2\|\bfw\|_{\frac{2q}{2-q}}
	+
	\left(\|\bfw+\bfw_0\|_4+c_4\lambda \|\nabla\bfw_0\|_{q}\right)\|\nabla\bfw\|_{\frac{4q}{4-q}}+c_4\lambda\|\nabla\bfw_0\|_q \|D^2\bfw\|_q\\ &\quad +c_4\lambda\|\nabla\bfw_0\|_q\|p\|_{\frac{3q}{3-q}}+c_4\lambda\|\nabla\bfw_0\|_q\|\nabla p\|_q\,.
	\end{aligned}
	\eeq{f_Lq}
	 By \cite[(II.6.22)]{Gab}, we  get the existence of some $\gamma_1>0$, independent of $p$ such that
	\be
	\|p\|_{\frac{3q}{3-q}}\le \gamma_1\|\nabla p\|_q\,.
	\eeq{p}
	Thus, plugging \eqref{trilinear_theta}-\eqref{f_Lq} in \eqref{Oseen}, we obtain 
		\be
		\begin{aligned}
&	a_1\|\bfw\|_{\frac{2q}{2-q}}+a_2\|\nabla\bfw\|_{\frac{4q}{4-q}}+\|D^2\bfw\|_{q}+\frac{1}{\gamma_1}\|p\|_{\frac{3q}{3-q}}+\|\nabla p\|_{q}
	\\&\le c_5\bigg(\lambda \|\nabla\bfw_0\|_2\|\bfw\|_{\frac{2q}{2-q}}
	+\left(\lambda\|\bfw+\bfw_0\|_4+\lambda^2 \|\nabla\bfw_0\|_{q}+\lambda\right)\|\nabla\bfw\|_{\frac{4q}{4-q}}+(\lambda^2\|\nabla\bfw_0\|_q+\lambda) \|D^2\bfw\|_q\\&\quad\quad+(\lambda^2\|\nabla\bfw_0\|_q+\lambda)\|p\|_{\frac{3q}{3-q}}+(\lambda^2\|\nabla\bfw_0\|_q+\lambda)\|\nabla p\|_q\bigg)\,.
	\end{aligned}
	\eeq{Oseen_2}
	 Since $q=\frac{4s}{4-s}$ , we can use Lemma \ref{Lq-estimate} and, under condition \eqref{smallness1}, we deduce that
	$$
		\begin{aligned}
	&a_1\|\bfw\|_{\frac{2q}{2-q}}+a_2\|\nabla\bfw\|_{\frac{4q}{4-q}}+\|D^2\bfw\|_{q}+\frac{1}{\gamma_1}\|p\|_{\frac{3q}{3-q}}+\|\nabla p\|_{q}\\&\le c_5\bigg(\frac{\lambda}{a_1}\left(1+\lambda\right)\,a_1\|\bfw\|_{\frac{2q}{2-q}}+\frac{1}{a_2}\big(\frac{\lambda}{a_1}(1+\lambda) +\frac{\lambda^2}{a_2}(1+\lambda)+\lambda\big)\,a_2\|\nabla\bfw\|_{\frac{4q}{4-q}}
	\\&\qquad+\big(\frac{\lambda^2}{a_2}(1+\lambda)+\lambda\big)(\|D^2\bfw\|_{q}+\|\nabla p\|_{q})
	+\gamma_1\big(\frac{\lambda^2}{a_2}(1+\lambda)+\lambda\big)\frac{1}{\gamma_1}\|p\|_{\frac{3q}{3-q}}
	\bigg)\,.
	\end{aligned}
	$$
Taking into account \eqref{aa} and the bound on $\lambda$, we deduce from this inequality that
	$$
	\begin{aligned}
		&a_1\|\bfw\|_{\frac{2q}{2-q}}+a_2\|\nabla\bfw\|_{\frac{4q}{4-q}}+\|D^2\bfw\|_{q}+\frac{1}{\gamma_1}\|p\|_{\frac{3q}{3-q}}+\|\nabla p\|_{q}\\&\le c_6\frac{\lambda}{a_1a_2}\left(a_1\|\bfw\|_{\frac{2q}{2-q}}+a_2\|\nabla\bfw\|_{\frac{4q}{4-q}}+\|D^2\bfw\|_{q}+\frac{1}{\gamma_1}\|p\|_{\frac{3q}{3-q}}+\|\nabla p\|_{q}\right)\,,
	\end{aligned}
	$$
	with yet another constant $c_6$ depending on $\Omega$ only.
Restricting again the value of $\lambda$ if necessary to ensure that 
	\be
	\frac\lambda{
	\text{min}\{1,\lambda^{1/2}\}\text{min}\{1,\lambda^{1/4}\}}<\frac{1}{c_6}
	\eeq{s2}
	we conclude that ${\bfw}=\bf0$ a.e. in $\Omega$. Therefore, Theorem \ref{th:unicita} by redefining $\lambda_0$ in \eqref{smallness1}, if necessary, to ensure \eqref{s2} holds. 

\bigskip\par
\noindent
{\bf Acknowledgements.} The work of G.P.~Galdi is partially supported by National Science Foundation Grant DMS-2307811.
D. Bonheure and C. Patriarca are supported by the WBI grant ARC Advanced 2020-25 ``PDEs in interaction'' at ULB. D. Bonheure is also partially supported by the FNRS PDR grant T.0020.25 and the Fonds Thelam 2025-F2150080-0021313.

\ed